\documentclass{amsart}
\usepackage{amsmath}
\usepackage{amsthm}
\usepackage{amssymb}
\usepackage{color}

\newtheorem{thm}{Theorem}
\newtheorem{ex}[thm]{Example}

\newtheorem{defi}[thm]{Definition}
\newtheorem{prop}[thm]{Proposition}
\newtheorem{rk}[thm]{Remark}

\newcommand{\rr}{{\mathbb{R}}}
\newcommand{\nn}{{\mathbb{N}}}
\newcommand{\e}{\varepsilon}
\newcommand{\vip}{\vskip.14cm}
\newcommand{\indiq}{{\bf 1}}
\newcommand{\dd}{{\rm d}}

\begin{document}

\title{On gelation for the Smoluchowski equation}

\author{Nicolas Fournier}
\address{Sorbonne Universit\'e, LPSM-UMR 8001, Case courrier 158, 75252 Paris Cedex 05, France.}
\email{nicolas.fournier@sorbonne-universite.fr}


\begin{abstract}
Motivated by the recent results of Andreis-Iyer-Magnanini \cite{aim}, we
provide a short proof, revisiting the one of Escobedo-Mischler-Perthame \cite{emp}, 
that for a large class of coagulation kernels,
any weak solution to the Smoluchowski equation looses mass in finite time. 
The class of kernels we consider is essentially the same as the one of \cite{aim}:
homogeneous kernels of degree $\gamma>1$ not vanishing on the diagonal, or 
homogeneous kernels of degree $\gamma=1$ not vanishing on the diagonal with some additional logarithmic factor.
We also show that when $\gamma=1$, the power of the logarithmic factor ensuring gelation
may depend on the shape of the kernel.
\end{abstract}

\maketitle

\section{Introduction}
The Smoluchowski coagulation equation \cite{s} describes the time-evolution of the concentration $f_t(x)$
of particles with mass $x\in (0,\infty)$ at time $t\geq 0$, in an infinite particle system. 
Assuming that two particles with masses $x$ and $y$ coalesce
at rate $K(x,y)$ to produce a single particle with mass $x+y$, it naturally writes 
\begin{equation}\label{ce}
\partial_t f_t(x)=\frac{1}{2} \int_0^x K(y,x-y) f_t(y)f_t(x-y)\dd y
- f_t(x) \int_0^\infty K(x,y)f_t(y)\dd y.
\end{equation}
We refer to the review papers of Aldous \cite{a} and Lauren\c cot \cite{l}
and to the books of Banasiak-Lamb-Lauren\c cot \cite{bll,bll2}, which contain a lot of information on this
equation and its discrete version.
We of course expect that \eqref{ce} preserves mass, {\it i.e} that
$\int_0^\infty x f_t(x)\dd x= \int_0^\infty x f_0(x)\dd x$ for all $t\geq 0$.
However, this is not true if the kernel $K$ increases sufficiently quickly for large
masses. Some mass may be lost in finite time, due to the appearance of clusters with infinite mass.
This phenomenon is called \textit{gelation}.

\vip

Gelation is easily established for the multiplicative kernel $K(x,y)=xy$, and physicists conjectured
that gelation occurs for $\gamma$-homogeneous kernels
(that is, $K(\lambda x,\lambda y)=\lambda^\gamma K(x,y)$) with $\gamma>1$.
See Ziff \cite{z}, Leyvraz-Tschudi \cite{lt} and Hendriks-Ernst-Ziff \cite{hez}.
Leyvraz \cite{ley} and then Buffet-Pul\'e \cite{bp} rigorously proved gelation when 
$f_0=\delta_1$ and $K(x,y)=r(x)\indiq_{\{x=y\}}$ is a {\it diagonal} kernel, with
$r(x) \geq x \log^{2+\e}(e+x)$ for some $\e>0$.
Jeon \cite{j} was the first to prove rigorously, through probabilistic arguments, the existence of a gelling
solution when $K(x,y)\geq (xy)^{\gamma/2}$ with $\gamma>1$.
Escobedo-Mischler-Perthame \cite{emp} found a simpler deterministic argument, stronger
since they showed that {\it any} solution is gelling. They also
proved many estimates on the rate of decay of the mass and the profile of the solution.
Lauren\c cot \cite[Proposition 36]{l} showed that any solution is gelling 
for a class of kernels including $K(x,y)=\sqrt{xy}[\log(e+x)\log(e+y)]^{1+\e}$ with $\e>0$.

\vip

Using a probabilistic approach in the spirit of \cite{j},
Andreis-Iyer-Magnanini \cite{aim} recently proved
that gelling solutions exist for a wider class of kernels. They actually deal with the slightly more general
model introduced by Norris \cite{n2}. Concerning \eqref{ce},
they are able to show gelation for any $\gamma$-homogeneous kernel non vanishing near the diagonal with $\gamma>1$.
They can also deal with e.g. $K(x,y)=(x\land y) \log^{3+\e}(e+x\land y)$
with $\e>0$ and others, see Example \ref{ex} below and the two lines after. Such kernels were
(really) not covered by \cite{j,emp,l}.

\vip

Our goal is to revisit the proof of \cite{emp} to try include the kernels treated in \cite{aim}.
To this end, we slightly simplify and refine the arguments of \cite{emp} and show,
with a very short proof, that {\it any} solution looses mass in finite time,
for a class of kernels including $K(x,y)=(x\land y) \log^{2+\e}(e+x\land y)$.

\vip
We also show that the optimality of the power $2$ of the logarithmic term is debatable:
we find two $1$-homogeneous kernels $K_0(x,y)=(x\land y)(\frac{x\land y}{x\lor y}-\frac12)_+$
and $K_1(x,y)=x+y$ 
such that when $K(x,y)=K_0(x,y) \log^\alpha(e+x\land y)$, weak solutions are gelling if and only if
$\alpha>2$ for the mono-disperse initial condition $f_0=\delta_1$,
while when $K(x,y)=K_1(x,y) \log^\alpha(e+x\land y)$, 
weak solutions are gelling if and only if $\alpha>1$.

\vip

Since $K_0(x,y)=0$ if $x/y\leq 1/2$, the support of the solution $f_t$ starting from $f_0=\delta_1$ 
is naturally included in $2^\nn$, and we have $K_0(2^n,2^m)=0$ if $n\neq m$. Thus everything
happens as if $K_0(x,y)=(x/2)\indiq_{\{x=y\}}$. Hence gelation when 
$K(x,y)=K_0(x,y) \log^\alpha(e+x\land y)$ with $\alpha>2$ and $f_0=\delta_1$ was already known \cite{bp},
but they showed non-gelation only when $\alpha \leq 1$.

\vip

The complexity of $1$-homogeneous kernel is not new:
van Dongen-Ernst \cite[Section 6]{vde} and Hermann-Niethammer-Vel\'azquez \cite{hnv} showed, {\it via} 
rigorous and heuristic results, that the large-time behavior of (non-gelling) solutions depends on the precise 
shape of the kernel.

\section{Main result and proof}

We always assume at least that $K:(0,\infty)^2 \to \rr_+$ is measurable and symmetric, {\it i.e.} 
$K(x,y)=K(y,x)$.
The following definition is more or less classical, see e.g. Lauren\c cot \cite[Theorem 24]{l}.

\begin{defi}\label{dwe}
A measurable family $(f_t)_{t\geq 0}$ of nonnegative measures on $(0,\infty)$ is said
to be a weak solution to \eqref{ce} if for all $t\geq 0$,
\begin{equation}\label{cond}
\int_0^t \int_0^\infty \int_0^\infty K(x,y) f_s(\dd y)f_s(\dd x) \dd s <\infty,
\end{equation}
and if for all bounded measurable $\psi : (0,\infty) \to \rr_+$, all $t\geq 0$,
\begin{equation}\label{we}
\int_0^\infty \psi(x) f_t(\dd x) = \int_0^\infty \psi(x) f_0(\dd x)+ \frac12
\int_0^t \int_0^\infty\int_0^\infty \Delta_\psi(x,y)
K(x,y) f_s(\dd y)f_s(\dd x) \dd s,
\end{equation}
where $\Delta_\psi(x,y)=\psi(x+y)-\psi(x)-\psi(y)$.
\end{defi}

All the terms make sense in \eqref{we}: the first and second one belong to $[0,\infty)\cup\{+\infty\}$
and the last one is finite by \eqref{cond} and since $|\Delta_\psi|$ is bounded. 
Moreover, for any bounded measurable  $\psi : (0,\infty) \to \rr_+$
for which  ${\int_0^\infty \psi(x) f_0(\dd x)<\infty}$, the map $t\to \int_0^\infty \psi(x) f_t(\dd x)$ is continuous
from $[0,\infty)$ to $\rr_+$.

\vip

Norris \cite[Theorem 4.1]{n} found a very general existence result: it suffices that $K$ 
is continuous on $(0,\infty)^2$ and that 
there exists a continuous function $\theta : (0,\infty)\to(0,\infty)$, sublinear (that is $\theta(\lambda x)
\leq \lambda \theta(x)$ if $\lambda \geq 1$)
such that $\int_0^\infty \theta(x) f_0(\dd x) <\infty$,
$K(x,y)\leq \theta(x)\theta(y)$ and $\lim_{(x,y)\to\infty}  [\theta(x)\theta(y)]^{-1}K(x,y)=0$.
This applies to many physical kernels, possibly diverging for small and large masses.

\vip
For $f$ a nonnegative measure on $(0,\infty)$ and $k\in \rr$, we set 
$M_k(f)=\int_0^\infty x^k f(\dd x)$. If applying \eqref{we} with $\psi(x)=x$,
we would find that $M_1(f_t)=M_1(f_0)$ for all $t\geq 0$, because $\Delta_\psi=0$. 
But this is not licit, since 
$\psi(x)=x$ is not bounded. However, the following observation is very classical.

\begin{rk}\label{tr} For any weak solution $(f_t)_{t\geq 0}$ to \eqref{ce},
we have $M_1(f_t)\leq M_1(f_0)$ for all $t\geq 0$.
\end{rk}

\begin{proof} Fix $t\geq 0$. For $a>0$, we set $\psi_a(x)=x\land a$ and observe that
$\Delta_{\psi_a}(x,y)\leq 0$ for all $x,y\in (0,\infty)$. By \eqref{we}, this implies 
that $\int_0^\infty (x\land a) f_t(\dd x) \leq  \int_0^\infty (x\land a) f_0(\dd x) \leq M_1(f_0)$.
The conclusion follows, by monotone convergence.
\end{proof}

Here is our main result. See Remark \ref{hom} and Example \ref{ex} below for applications.

\begin{thm}\label{mr} Assume that
\begin{gather}\label{c2}
\hbox{there exist $x_0>0$ and $r>1$ such that }\;\;
\kappa:=\int_{x_0}^\infty [H(a)]^{-1/2}\dd a <\infty,\\
\hbox{where}\;\;H(a)=a \inf\{K(x,y) : x,y\in [a,r a]\}.\notag
\end{gather}
For any weak solution $(f_t)_{t\geq 0}$ to \eqref{ce} with  
$M_1(f_0)<\infty$ and $f_0((x_0,\infty))>0$, we have
$$
T_{gel}:=\inf\{t\geq0 : M_1(f_t)<M_1(f_0)\} \leq 2 \kappa^2 \Big(\frac r{r-1}\Big)^2 
\frac{M_1(f_0)}{[\int_{x_0}^\infty (x-x_0)f_0(\dd x)]^2}.
$$
\end{thm}

\begin{proof}
For $a>0$, we set $\psi_a(x)=x\land a$. We have $\Delta_{\psi_a}(x,y)\leq 0$ for all $x,y>0$ and
$\Delta_{\psi_a}(x,y)=-a$ when $x,y\geq a$. Hence
$$
\Delta_{\psi_a}(x,y)K(x,y) \leq -a K(x,y)\indiq_{\{x,y\geq a\}}\leq -a K(x,y) \indiq_{\{x,y \in [a,r a]\}}\leq 
-H(a)\indiq_{\{x,y \in [a,r a]\}}.
$$
Applying \eqref{we} with $\psi_a$, we find that for all $t\geq 0$,
\begin{equation}\label{uez}
\int_0^\infty (x \land a) (f_t(\dd x)-f_0(\dd x)) \leq - \frac{H(a)}2\int_0^t \int_a^{r a}\int_a^{r a}
f_s(\dd y)f_s(\dd x) \dd s=- \frac{H(a)}2\int_0^t [F_s(a)]^2\dd s,
\end{equation}
where $F_s(a)=f_s([a,ra])$. Since $f_t$ is nonnegative, we conclude that
\begin{equation*}
\int_0^t H(a)[F_s(a)]^2 \dd s \leq  2 \int_0^\infty (x\land a) f_0(\dd x)\leq 2 M_1(f_0).
\end{equation*}
Multiplying this inequality by $[H(a)]^{-1/2}$ and integrating in  $a\in[x_0,\infty)$, using \eqref{c2},
we find
\begin{equation}\label{tac1}
\int_0^t \int_{x_0}^\infty [H(a)]^{1/2}[F_s(a)]^2 \dd a \dd s \leq  2 \kappa M_1(f_0).
\end{equation}
We next write $M_1(f_s)=A_s+B_s$, where 
$$
A_s=\int_0^{\infty}(x\land x_0)f_s(\dd x)\quad \hbox{and} \quad B_s=\int_{x_0}^\infty(x-x_0)f_s(\dd x).
$$
We have $A_s\leq A_0$ by \eqref{uez} (with $a=x_0$). Morover,
since for all $x\geq x_0$,
$$
\int_{x_0}^\infty \indiq_{\{x \in [a,ra]\}} \dd a 
= x-\max\Big\{x_0,\frac xr\Big\}
=\min\Big\{x-x_0,\frac{r-1}{r}x\Big\}\geq \frac{r-1}{r}(x-x_0),
$$ 
we have
$$
\int_{x_0}^\infty F_s(a) \dd a=\int_{x_0}^\infty \int_{a}^{ra} f_s(\dd x)\dd a
=\int_{x_0}^\infty \int_{x_0}^\infty \indiq_{\{x \in [a,ra]\}} \dd a f_s(\dd x)
\geq \frac {r-1}r B_s.
$$
Using the Cauchy-Schwarz inequality and \eqref{c2}, we get
\begin{equation*}
B_s^2\leq \Big(\frac r {r-1}\Big)^2
\Big[\int_{x_0}^\infty [H(a)]^{-1/4} [H(a)]^{1/4}F_s(a)\dd a \Big]^2
\leq \Big(\frac r {r-1}\Big)^2 \kappa \int_{x_0}^\infty [H(a)]^{1/2}[F_s(a)]^2 \dd a.
\end{equation*}
Integrating this inequality in $s\in [0,t]$ and using \eqref{tac1}, we conclude that for all $t\geq 0$,
\begin{equation}\label{ib}
\int_0^t B_s^2 \dd s\leq 2 \Big(\frac r {r-1}\Big)^2 \kappa^2 M_1(f_0).
\end{equation}
We have shown that $M_1(f_s)\leq A_0+B_s$, with $A_0=\int_0^{\infty}(x\land x_0)f_0(\dd x)$.
For all $s\in [0,T_{gel})$, we have $M_1(f_s)=M_1(f_0)$ and thus  
$B_s \geq M_1(f_0)-A_0=\int_{x_0}^\infty (x-x_0)f_0(\dd x)>0$, since  $f_0((x_0,\infty))>0$ by assumption.
Inserted in \eqref{ib} (with $t=T_{gel}$), this gives
$$
T_{gel} \Big[\int_{x_0}^\infty (x-x_0)f_0(\dd x)\Big]^2\leq 
2 \Big(\frac r {r-1}\Big)^2\kappa^2 M_1(f_0).
$$
\vskip-.7cm
\end{proof}
\vip\vip

\begin{rk}\label{hom}
Assume that $K:(0,\infty)^2 \to \rr_+$ is continuous, $\gamma$-homogeneous (i.e.
$K(\lambda x,\lambda y)=\lambda^\gamma K(x,y)$ for all $\lambda,x,y>0$) 
with $\gamma>1$, and that $K(1,1)>0$. Then \eqref{c2} is met
for any $x_0>0$.
\end{rk}

\begin{proof} By continuity, there is $\e>0$ such that
$$
\rho=\inf\{K(x,y) : x,y \in [1,1+\e]\}>0.
$$
Choosing $r=1+\e$ and using the homogeneity of $K$, one can check that
$$
H(a)=a \inf\{K(x,y) : x,y \in [a,ra]\}=a^{1+\gamma} \inf\{K(x,y) : x,y \in [1,r]\}  \geq \rho a^{1+\gamma}.
$$
Since $\gamma>1$, the conclusion follows.
\end{proof}

\begin{ex}\label{ex}
Fix $\gamma>1$, $\e>0$ and $\theta \geq 0$. If $K(x,y)\geq K_0(x,y)$, with (here $u_+=\max\{u,0\}$)
\begin{gather*}
K_0(x,y)=  (x\land y)^\gamma \Big(\frac{x\land y}{x\lor y}\Big)^{\theta} 
\quad \hbox{or}\quad  K_0(x,y)=  (x\land y) \Big(\frac{x\land y}{x\lor y}\Big)^{\theta}\log^{2+\e}(1+x\land y),\\
K_0(x,y)=  (x\land y)^\gamma \Big(\frac{x\land y}{x\lor y}-\frac 12\Big)_+^\theta
\quad \hbox{or}\quad K_0(x,y)=  (x\land y) 
\Big(\frac{x\land y}{x\lor y}-\frac 12\Big)_+^\theta\log^{2+\e}(1+x\land y),
\end{gather*}
then \eqref{c2} is met: any choice of $x_0>0$ and $r>1$ is suitable for the two first kernels, 
and any choice of $x_0>0$ and $r\in (1,2)$ is convenient for the two last ones.
\end{ex}

None of these four kernels are not covered by \cite{j,emp,l}. The first and third ones are covered by \cite{aim},
and they are able to treat the second and fourth ones, with a power $3+\e$ instead of $2+\e$
for the logarithmic term. Hence the improvement of the present work is not huge when compared to \cite{aim}.
However, we can show that {\it any} weak solution is gelling for those kernels, and the proof is very short.
Let us mention that we are far from being able to deal with kernels vanishing on the diagonal,
such as $K(x,y)=(x^{1/3}+y^{1/3})^2|x^{2/3}-y^{2/3}|$ mentioned in \cite[Table 1]{a}.

\section{About criticality}

By \cite[Propositions 33 and 36]{l}, 
the kernel $K(x,y)=(xy)^{1/2}\log^{\alpha/2}(e+x)\log^{\alpha/2}(e+y)$ is non-gelling if $\alpha \leq 1$,
but gelling if $\alpha>2$. However, we cannot decide what happens when $\alpha \in (1,2]$.
The same applies to $K(x,y)=(x\land y)\log^{\alpha}(e+x\land y)$.

\vip

All this concerns kernels that are $1$-homogeneous, that is $K(\lambda x,\lambda y)=\lambda K(x,y)$, 
up to some logarithmic factor. In Propositions \ref{log1} and \ref{log2} below, we find two 
$1$-homogeneous kernels 
$K_0$ and $K_1$ such that $K(x,y)=K_0(x,y)\log^{\alpha}(e+x\land y)$ is gelling if and only if $\alpha>2$
({\bf only when $f_0=\delta_1$}),
while $K(x,y)=K_1(x,y)\log^{\alpha}(e+x\land y)$ is gelling if and only if $\alpha>1$.
Hence  the critical exponent of the logarithmic factor may
depend on the precise shape of the kernel.

\begin{prop}\label{log1} Consider
$$
K(x,y)=(x \land y) \Big(\frac{x\land y}{x\lor y}-\frac12\Big)_+\log^\alpha (e+ x\land y).
$$

(i) If $\alpha>2$, any weak solution to \eqref{ce} such that $M_1(f_0)\in (0,\infty)$ is gelling.

\vip

(ii) If $\alpha\in (0,2]$, then any weak solution $(f_t)_{t\geq 0}$ to \eqref{ce} starting from $f_0=\delta_1$
is non-gelling.
\end{prop}

\begin{proof}
Point (i) follows from Theorem \ref{mr}: with the choice $r=3/2$, we have 
$H(a)=\frac16 a^2 \log^{\alpha}(e+a)$. Hence if $\alpha>2$, $\int_{x_0}^\infty [H(a)]^{-1/2}\dd a <\infty$
for all $x_0>0$. For (ii), we fix $\alpha \in (0,2]$ and a weak solution $(f_t)_{t\geq0}$
issued from $f_0=\delta_1$ and show in several steps the absence of gelation

\vip

{\it Step 1.} Here we prove that for any $t\geq 0$, $f_t((0,\infty)\setminus 2^\nn)=0$, 
where $2^\nn=\{2^n : n \in \nn\}$. 

\vip

To this end, we define $A_0=(0,1)$ and, for $n\geq 0$, $A_{n+1}=A_n \cup
(2^n,2^{n+1})$ and we 
show by induction that $f_t(A_n)=0$ for all $t\geq 0$.

\vip

When $n=0$, we apply \eqref{we} with $\psi=\indiq_{(0,1)}$. Since $\Delta_\psi(x,y)\leq 0$ for all $x>0,y>0$,
we conclude that $f_t((0,1))\leq f_0((0,1))=0$ for all $t\geq 0$.

\vip

We next fix $n\in \nn$ and assume that $f_t(A_n)=0$, so that
$f_t(\dd x)=\sum_{k=0}^n f_t(\{2^k\})\delta_{2^k}(\dd x) + \indiq_{\{x\geq 2^n\}}f_t(\dd x)$, for all $t\geq 0$. 
We want to 
show that $f_t((2^n,2^{n+1}))=0$ for all $t\geq 0$.
We apply \eqref{we} with $\psi=\indiq_{(2^n,2^{n+1})}$.
Using that $f_0((2^n,2^{n+1}))=0$ and ignoring the nonpositive terms, this gives
\begin{align*}
f_t((2^n,2^{n+1})) \leq& \frac 12 
\int_0^t \int_0^\infty\int_0^\infty \indiq_{\{x+y \in (2^n,2^{n+1})\}}K(x,y) f_s(\dd y)f_s(\dd x) \dd s
=I^n_t+2J^n_t+K^n_t,
\end{align*}
where
\begin{align*}
I^n_t=&\frac12 \sum_{k,\ell=0}^n \int_0^t \indiq_{\{2^k+2^\ell\in (2^n,2^{n+1})\}}
K(2^k,2^\ell) f_s(\{k\})f_s(\{\ell\}) \dd s,\\
J^n_t=& \frac12 \sum_{k=0}^n \int_0^t \int_{y\geq 2^n} \indiq_{\{2^k+y \in (2^n,2^{n+1})\}}
K(2^k,y) f_s(\dd y)f_s(\{k\}) \dd s,\\
K^n_t=&\frac 12 \int_0^t \int_{x\geq 2^n}\int_{y\geq 2^n}\indiq_{\{x+y \in (2^n,2^{n+1})\}}K(x,y) f_s(\dd y)f_s(\dd x) \dd s.
\end{align*}
First, $I^n_t=0$, since $K(2^k,2^\ell)=0$ for all $k\neq \ell$ and
and since $2^k+2^k \notin (2^n,2^{n+1})$ for all $k=0,\dots,n$. Next, $J^n_t=0$,
because $y \geq  2^n$ and $2^k+y <2^{n+1}$ imply that $k \leq n-1$, so that $2^k\leq y/2$ and thus $K(2^k,y)=0$.
Finally, $K^n_t=0$ since $x\geq 2^n$ and $y\geq 2^n$ imply that $x+y \geq 2^{n+1}$.

\vip

{\it Step 2.} By Step 1, we may write $f_t=\sum_{n\geq 0} f_t(\{2^n\}) \delta_{2^n}$.
We show here that for each $n\in\nn$, $t\mapsto  f_t(\{2^n\})$ is $C^1$ on $[0,\infty)$ and 
$$
\frac{\dd}{\dd t}f_t(\{2^n\})=-K(2^n,2^n)[f_t(\{2^n\})]^2+ \frac12 K(2^{n-1},2^{n-1})[f_t(\{2^{n-1}\})]^2 
\indiq_{\{n\geq 1\}}.
$$

If first $n\geq 1$, we apply \eqref{we} with $\psi(x)=\indiq_{\{x=2^{n}\}}$ and get,
since $f_0(\{2^{n}\})=0$,
$$
f_t(\{2^n\})=\frac12 \int_0^t \sum_{k,\ell \in \nn} [\indiq_{\{2^k+2^\ell=2^n\}}-\indiq_{\{2^k=2^n\}}-\indiq_{\{2^\ell=2^n\}}]
K(2^k,2^\ell)f_s(\{2^k\})f_s(\{2^\ell\}) \dd s.
$$
Since $K(2^k,2^\ell)=0$ for $k\neq \ell$, we conclude that
$$
f_t(\{2^n\})=\frac12 \int_0^t K(2^{n-1},2^{n-1})[f_s(\{2^{n-1}\})]^2\dd s-  
\int_0^t K(2^n,2^n)[f_s(\{2^n\})]^2 \dd s.
$$
Since $t\mapsto f_t(\{2^k\})$ is continuous for all $k\in \nn$ by Definition \ref{dwe},
the conclusion follows. If next $n=0$, we apply \eqref{we} with $\psi(x)=\indiq_{\{x=1\}}$ and get,
since $f_0(\{1\})=1$,
\begin{align*}
f_t(\{1\})=&1+\frac12 \int_0^t \sum_{k,\ell \in \nn} [\indiq_{\{2^k+2^\ell=1\}}-\indiq_{\{2^k=1\}}-\indiq_{\{2^\ell=1\}}]
K(2^k,2^\ell)f_s(\{2^k\})f_s(\{2^\ell\}) \dd s\\
=&1 -  \int_0^t K(1,1)[f_s(\{1\})]^2 \dd s
\end{align*}
as desired. We used that $K(1,2^k)=0$ for all $k\geq 1$.

\vip

{\it Step 3.}
We now prove that for any $n\in \nn$, 
$$
b_n:=\sup_{t\geq 0} f_t(\{2^n\}) \leq \frac{\log^{\alpha/2} (e+1)}{2^{n}\log^{\alpha/2}(e+2^n)}.
$$

We proceed by induction. If $n=0$, we know from Step 2 that $t\mapsto f_t(\{2^0\})$ is non-increasing,
whence $b_0 = f_0(\{2^0\})=1$. 
Fix $n\geq 1$ and assume that $b_{n-1}$ satisfies the desired estimate. Then
$$
\frac{\dd}{\dd t}f_t(\{2^{n}\})\leq -K(2^n,2^n)[f_t(\{2^n\})]^2+ \frac12 K(2^{n-1},2^{n-1})b_{n-1}^2
$$ 
by Step 2.
Since $f_0(\{2^{n}\})=0$, we classically conclude that
$$
b_n \leq \Big(\frac{ K(2^{n-1},2^{n-1})}{2K(2^n,2^n)} \Big)^{1/2}b_{n-1}
=\frac 12 \frac{\log^{\alpha/2}(e+2^{n-1})}{\log^{\alpha/2}(e+2^n)} b_{n-1}
\leq \frac{\log^{\alpha/2} (e+1)}{2^n\log^{\alpha/2}(e+2^n)}.
$$

{\it Step 4.} We show  that there are $A,B>0$ such that
for any $t>0$, $\int_0^\infty x \log(e+x) f_t(\dd x) \leq A e^{B t}$.

\vip

We fix $a>1$ and apply \eqref{we} with $\psi_a(x)=x\log(e+x) \indiq_{\{x\leq a\}}$. By Step
1 and since $K(2^k,2^n)=0$ if $k\neq n$, we find, introducing $A=\log (e+1)$,
\begin{align*}
\int_0^a x \log(1+x) f_t(\dd x) =& A + \frac12 \sum_{n\geq 0}\int_0^t \Delta_{\psi_a}(2^n,2^n)
K(2^n,2^n)f_s(\{2^n\}) f_s(\{2^n\})  \dd s.
\end{align*}
But
$$
\Delta_{\psi_a}(x,x)=\psi_a(2x)-2\psi_a(x)\leq \indiq_{\{x\leq a/2\}} 2x \log\Big(\frac{e+2x}{e+x}\Big)
\leq \indiq_{\{x\leq a\}} 2x \log2,
$$
whence,
\begin{align*}
\int_0^a x \log(e+x) f_t(\dd x) \leq& A + \log 2 \sum_{n\geq 0}\int_0^t 2^n \indiq_{\{2^n\leq a\}}
K(2^n,2^n)f_s(\{2^n\}) f_s(\{2^n\})  \dd s.
\end{align*}
Using the expression of $K(2^n,2^n)$ and  Step 3, we find
$$
K(2^n,2^n)f_s(\{2^n\})\leq 2^n \log^\alpha(e+2^n)  b_n\leq 
\log^{\alpha/2} (e+1) \log^{\alpha/2}(e+2^n)\leq \log^{\alpha-1} (e+1) \log(e+2^n),
$$
since $\alpha \in (0,2]$ (so that $0<u<v$ implies $u^{\alpha/2}v^{\alpha/2}\leq u^{\alpha-1}v$). 
Setting $B=  (\log 2) (\log^{\alpha-1} (e+1))$,
\begin{align*}
\int_0^a x\log(e+x) f_t(\dd x) \leq&  A + B \sum_{n\geq 0} \int_0^t 
\indiq_{\{2^n\leq a\}} 2^n \log(e+2^n)f_s(\{2^n\}) \dd s\\
= & A+B  \int_0^t \int_0^a x\log (e+x)f_s(\dd x) \dd s
\end{align*}
by Step 1 again. Thus $\int_0^a x\log(e+x) f_t(\dd x)\leq A e^{Bt}$ by the
Gronwall lemma, whence the result.
\vip

{\it Step 5.} We finally conclude that $M_1(f_t)=M_1(f_0)$ for all $t\geq 0$. 
By \eqref{we} with $\psi_a(x)=x\land a$,
$$
\int_0^\infty (x\land a) f_t(\dd x)=\int_0^\infty (x\land a)f_0(\dd x) + I_a(t),
$$
where
$$
I_a(t)=\int_0^t \int_0^\infty \int_0^\infty \Delta_{\psi_a(x,y)}K(x,y) f_s(\dd y)f_s(\dd x) \dd s.
$$
By monotone convergence, $\lim_{a\to \infty} \int_0^\infty (x\land a) f_t(\dd x)=M_1(f_t)$.
It thus suffices to check that $\lim_{a\to \infty} I_a(t)=0$ for each $t\geq0$.
For this we use dominated convergence:
since $\lim_{a\to \infty} \Delta_{\psi_a}(x,y)=0$ for all $x,y\in (0,\infty)$,
and since 
$$
\sup_{a>0}|\Delta_{\psi_a(x,y)}|=\sup_{a>0}(x\land a + y \land a -(x+y)\land a) 
\leq \sup_{a>0} ([x\land a] \land [y \land a])
= x\land y, 
$$
we only need that for all $t\geq0$,
\begin{equation}\label{tbru}
\int_0^t \int_0^\infty \int_0^\infty (x\land y) K(x,y)f_s(\dd y) f_s(\dd x) \dd s<\infty.
\end{equation}
But 
$$
(x\land y) K(x,y) \leq (x\land y)^2 \log^{\alpha} (e+x\land y)
\leq  (x\land y)^2 \log^{2} (e+x\land y)
\leq  x\log(e+x) \, y\log(e+y).
$$
Thus \eqref{tbru} follows from Step 4.
\end{proof}

\begin{prop}\label{log2}
(i) If $K(x,y)\leq (x + y) \log (e+ x\land y)$, then
any weak solution $(f_t)_{t\geq 0}$ to \eqref{ce} such that $M_1(f_0)+M_2(f_0)<\infty$
is non-gelling.
\vip
(ii) If $K(x,y)\geq (x + y) \log^\alpha (e+ x\land y)$ with some $\alpha>1$,
then any weak solution to \eqref{ce} such that $M_1(f_0)\in (0,\infty)$ and
$\int_0^1 x|\log x|\indiq_{\{x< 1\}} f_0(\dd x)<\infty$ is gelling.
\end{prop}

\begin{proof} We first prove point (i) in three steps.\vip

{\it Step 1.} For any nonnegative measure $f$ on $(0,\infty)$,
any $a>0$, any $A\geq M_1(f)$,
$$
\int_0^a y \log(e+y) f(\dd y) \leq A \log\Big(e+\frac1A\int_0^a y^2 f(\dd y)\Big).
$$
Indeed, the Jensen inequality applied to the probability measure 
$\mu_a(\dd y)=Z_a^{-1}y\indiq_{\{y \in (0,a]\}}f(\dd y)$,
with $Z_a=\int_0^a y f(\dd y)$, and the concave function $\log(e+y)$, directly gives
$$
\int_0^a y \log(e+y) f(\dd y) \leq Z_a \log\Big(e+\frac1{Z_a}\int_0^a y^2 f(\dd y)\Big).
$$
It then suffices to observe that $Z_a \leq M_1(f)\leq A$ and that for $b>0$ fixed,
the function $\varphi(z)= z \log(e+z^{-1}b)$ is nondecreasing on $[0,\infty)$,
because $\varphi''(z)=\frac{-b^2}{z(b+ez)^2}\leq 0$ and $\lim_{z\to \infty}\varphi'(z)=0$.

\vip
{\it Step 2.} We next show that $\sup_{[0,T]} M_2(f_t)<\infty$ for all $T>0$.
We apply \eqref{we} with $\psi_a(x)=x^2 \indiq_{\{x\leq a\}}$. Observing that
$\Delta_{\psi_a}(x,y) \leq \indiq_{\{x+y\leq a \}} 2xy\leq 2  \indiq_{\{x\leq a,y\leq a \}} xy$, we find
\begin{align*}
\int_0^a x^2 f_t(\dd x) \leq& \int_0^a x^2 f_0(\dd x) + \int_0^t \int_0^a\int_0^a
xy (x+y)\log(e+x\land y) f_s(\dd y)f_s(\dd x)\dd s\\
\leq & M_2(f_0)+ 2 \int_0^t \int_0^a\int_0^a x^2 y\log(e+y) f_s(\dd y)f_s(\dd x)\dd s
\end{align*}
by symmetry. Using Step 1 with $A=M_1(f_0)\geq M_1(f_s)$, recall Remark \ref{tr}, we find
\begin{align*}
\int_0^a x^2 f_t(\dd x) \leq& M_2(f_0)+ 2M_1(f_0) \int_0^t \Big(\int_0^ax^2 f_s(\dd x)\Big) 
\log\Big(e+ \frac1{M_1(f_0)} \int_0^a y^2 f_s(\dd y)\Big)\dd s.
\end{align*}
This classically implies that for any $T>0$, there is a constant $C_T$ depending only on
$T$, $M_2(f_0)$ and $M_1(f_0)$, such that $\sup_{[0,T]} \int_0^a x^2 f_t(\dd x)\leq C_T$.
The conclusion follows, letting $a\to \infty$ by monotone convergence.
\vip

{\it Step 3.} Exactly as in Step 5 of the proof of Proposition \ref{log1}, the only difficulty 
to show that $M_1(f_t)=M_1(f_0)$ for all $t\geq 0$ is to show \eqref{tbru}. But,
since $\log(e+u)\leq 1+u$ for all $u\geq 0$,
$$
(x\land y) K(x,y) \leq 2xy \log(e+x\land y)\leq 2xy(1+x\land y) \leq 2(x^2+x) y.
$$
Hence for all $t>0$,
$$
\int_0^t \int_0^\infty\int_0^\infty (x\land y) K(x,y) f_s(\dd y)f_s(\dd x)\dd s \leq 2\int_0^t (M_2(f_s)+M_1(f_s))
M_1(f_s)\dd s <\infty
$$ 
by Step 2 and Remark \ref{tr}.

\vip

Point (ii) does not follow from Theorem \ref{mr}.
The proof below, adapted from \cite[Proposition 2.3]{fl}, is completely different
and generally less powerful, since it allows to prove gelation when $K(x,y)=x y^{\gamma-1}+ x^{\gamma-1} y$ 
with $\gamma>1$,
but not when $K(x,y)=(x y)^{\gamma/2}$ with $\gamma\in (1,2)$ neither when $K(x,y)=x^\alpha y^\beta+x^\beta y^\alpha$
with $\alpha,\beta \in (0,1)$ and $\gamma=\alpha+\beta>1$.
We fix $a>0$ and consider 
$$
\psi_a(x)=x \indiq_{\{x\leq a\}}\int_x^\infty \frac{\dd u}{\frac u2\log^\alpha(e+\frac u2)}\geq 0.
$$
Using that $\alpha>1$, one can check that $\psi_a(x)\leq \kappa x (1+|\log x|\indiq_{\{x<1\}})$,
for some constant $\kappa$ not depending on $a$.
We have $\Delta_{\psi_a}(x,y)\leq 0$ if $x>a$ or $y>a$, so that for all $x,y>0$,
\begin{align*}
\Delta_{\psi_a}(x,y) \leq& \indiq_{\{x,y\leq a\}} 
\Big((x+y)\int_{x+y}^\infty \frac{\dd u}{\frac u2 \log^\alpha(e+\frac u2)}-
x\int_x^\infty \frac{\dd u}{\frac u2 \log^\alpha(e+\frac u2)} 
-y\int_y^\infty \frac{\dd u}{\frac u2 \log^\alpha(e+\frac u2)} \Big)\\
= & -\indiq_{\{x,y\leq a\}} \Big(x\int_x^{x+y} \frac{\dd u}{\frac u2 \log^\alpha(e+\frac u2)} 
+y\int_y^{x+y} \frac{\dd u}{\frac u2 \log^\alpha(e+\frac u2)} \Big)\\
\leq & -\indiq_{\{x, y\leq a\}} (x\land y) \int_{x\land y}^{x+y} \frac{\dd u}{\frac u2 \log^\alpha(e+\frac u2)}\\
\leq & -\indiq_{\{x, y\leq a\}} (x\land y) \int_{x\land y}^{2(x\land y)} \frac{\dd u}{\frac u2 \log^\alpha(e+\frac u2)}\\
\leq & -\indiq_{\{x,y\leq a\}} \frac{x\land y}{\log^\alpha(e+x\land y)}.
\end{align*}
Thus $K(x,y)\Delta_{\psi_a}(x,y) \leq -\indiq_{\{x,y\leq a\}} (x\land y)(x+y)\leq -\indiq_{\{x,y\leq a\}} xy$,
so that \eqref{we} gives 
$$
\int_0^\infty \psi_a(x)f_t(\dd x) \leq \int_0^\infty \psi_a(x)f_0(\dd x)-\frac12 \int_0^t \int_0^a \int_0^a 
xy f_s(\dd y)f_s(\dd x) \dd s.
$$
Since the left hand side is nonnegative and  $\psi_a(x)\leq \kappa x (1+|\log x|\indiq_{\{x<1\}})$, we conclude that 
$$
\int_0^t \int_0^a \int_0^a xy f_s(\dd y)f_s(\dd x) \dd s
\leq 2\kappa \int_0^\infty x(1+|\log x|\indiq_{\{x<1\}}) f_0(\dd x) =:C.
$$
Letting $a\to \infty$, we end with
$\int_0^t [M_1(f_s)]^2 \dd s \leq C$,
which implies that $T_{gel} \leq C/[M_1(f_0)]^2$.
\end{proof}

\end{document}